\documentclass[12pt]{amsart}
\usepackage[]{amsmath, amsthm, amsfonts, verbatim, amssymb, tikz,inputenc}

\usepackage{amsmath}
\usepackage{amsthm}
\usepackage{amssymb}
\usepackage{amsfonts}
\usepackage{hyperref}
\usepackage{rotating}
\usepackage{amssymb}
\usepackage{epstopdf}
\usepackage{pifont}
\usepackage{amsmath}
\usepackage{enumerate}
\usepackage{graphicx}
\usepackage{color}
\usepackage{tikz-cd}
\usepackage{tikz}
\usepackage{mathtools}
\usepackage{longtable}
\usepackage{booktabs}
\usepackage{comment}
\usepackage{textcomp}

\DeclareMathAlphabet{\mathpzc}{OT1}{pzc}{m}{it}

\usepackage{tikz}
\usetikzlibrary{arrows}

\newtheorem{theorem}{Theorem}[section]
\newtheorem{lemma}[theorem]{Lemma}
\newtheorem{proposition}[theorem]{Proposition}

\theoremstyle{definition}
\newtheorem{definition}[theorem]{Definition}
\theoremstyle{remark}
\newtheorem{remark}[theorem]{Remark}

\newtheorem*{Acknowledgments}{Acknowledgments}

\newcommand{\ord}{{\rm ord}}

\newcommand{\mult}{{\rm mult}}

\newcommand{\Proj}{\operatorname{Proj}}

\newcommand{\Spec}{{\rm Spec}}

\newcommand{\Supp}{{\rm Supp}}
\newcommand{\Sing}{{\rm Sing}}

\def\<{\langle}
\def\>{\rangle}

\title{Local inequalities for $cD$ and $cE$ singularities%Birational rigidity of Fano varieties containing a $cD_4$ singularity
}

\author{Chi-Kang Chang}
\address{\rm RIKEN Center for Interdisciplinary Theoretical and Mathematical Sciences (iTHEMS), 2-1 Hirosawa, Wako, Saitama 351-0198, Japan}
\email{chi-kang.chang@riken.jp}

\author{Jheng-Jie Chen}
\address{\rm Department of Mathematics, National Central University, Taoyuan City, 320, Taiwan}
\email{jhengjie@math.ncu.edu.tw}

\author{Takuzo Okada}
\address{\rm Faculty of Mathematics, Kyushu University, Fukuoka 819-0395, Japan}
\email{tokada@math.kyushu-u.ac.jp}

\begin{document}

\maketitle
\begin{abstract}
We obtain inequalities for isolated $cD_n, cE_6, cE_7$ and $cE_8$ singularities, analogues of the $4\mu^2/(n+1)$-inequality for isolated $cA_n$ singularity in \cite{KOP}. These inequalities are sharp. This enables us to prove the birational rigidity of certain families of Fano $3$-fold weighted hypersurfaces, which contain only terminal quotient singularities and one $cD_n$ singularity.   
\end{abstract}

%%%%%%%%%%%%%%%%%%%%%%%%%%%%%%%%%%
%%%%%%%%%%%%%%%%%%%%%%%%%%%%%%%%%%
\section{Introduction} \label{sec:intro}
%%%%%%%%%%%%%%%%%%%%%%%%%%%%%%%%%%
%%%%%%%%%%%%%%%%%%%%%%%%%%%%%%%%%%

A Fano variety is said to be birationally rigid if it has a unique Mori fiber structure up to birational equivalence. 
Using the Noether-Fano method, a birational non-rigid Fano variety $X$ yields a mobile linear system $\mathcal{M}\subset |-\mu K_X|$ for some rational number $\mu$ satisfying that the pair $(X,\frac{1}{\mu} \mathcal{M})$ has non-canonical singularities at a certain subvariety $Z\subset X$. When $\dim X=3$, $Z$ may be a curve or a point $P$. Suppose that $P\in X$ is non-singular. The $4\mu^2$-inequality is well-known and frequently used to exclude divisorial contractions whose centers are non-singular $3$-fold germs $P\in X$ in the Sarkisov program. It plays a crucial role in proving birational rigidity of quasismooth Fano $3$-fold weighted hypersurfaces of index $1$ (see \cite{CP}, \cite{CPR}).   
%\begin{theorem}[\cite[Corollary~3.4]{Corti}, and \cite[Theorem 2.1]{Puk13}]\label{4mu2ineq} Suppose that $P\in X$ is a germ of a smooth $3$-fold such that $P$ is a non-canonical singularity of $(X,\frac{1}{\mu} \mathcal{M})$. Then for general members $D_1, D_2\in \mathcal{M}$, one has \[ \mult_P (D_1 \cap D_2) >4\mu^2.\]
%\end{theorem}  

 Recently, the last author, Krylov, Paemurru and Park extended the $4\mu^2$-inequality to the case where $P\in X$ is a germ of a compound Du Val of type $cA_1$ in \cite[Theorem 1.2]{KOPP}. Then, a generalization for $cA_n$ was established as follows where $n$ is a positive integer:
\begin{theorem}[={\cite[Theorem 1.3]{KOP}}]\label{cAnineq} Suppose that $P\in X$ is a germ of a $cA_n$ point such that $P$ is a non-canonical singularity of $(X,\frac{1}{\mu} \mathcal{M})$, where $\mu > 0$ is a rational number and $\mathcal{M}$ is a mobile linear system of Cartier divisors on $X$. Then for general members $D_1, D_2\in \mathcal{M}$, one has \[ \mult_P (D_1 \cdot D_2) >\frac{4}{n+1}\mu^2.\]
\end{theorem}
 In \cite{KOPP, KOP}, the authors applied Theorem \ref{cAnineq} to obtain the birational rigidity for many families of Fano 3-fold weighted complete intersections with at worst terminal quotient and isolated $cA_n$ singularities.  
Furthermore, they extended the birational rigidity for sextic double solids which contain at worst $cA_1$ and ordinary $cA_2$ singularities as well.

Recall that a Gorenstein terminal $3$-fold germ was classified as a compound Du Val of type $cA_n, cD_n (\textup{for }n\geq 4), cE_6, cE_7$ and $cE_8$ (see \cite{Rei80, Rei83}). It is then natural to ask the analogies of Theorem \ref{cAnineq} for the cases $cD_n, cE_6, cE_7,$ and $cE_8$.

The aim of this article is to establish the following inequalities: 
\begin{theorem}[=Theorem \ref{3foldlocalineq}] \label{intro3foldlocalineq}
    Let $P \in X$  be a germ of a 3-dimensional terminal singularity of type either $cD_n$ or $cE_n$, $\mathcal{M}$ a mobile linear system of Cartier divisors on $X$ and $\mu$ a positive rational number.
    Suppose that $P$ is a center of non-canonical singularities of the pair $(X, \frac{1}{\mu} \mathcal{M})$.
    Then, for general members $D_1, D_2 \in \mathcal{M}$, we have
    \[
    \mult_P (D_1 \cdot D_2) >
    \begin{dcases}
        \frac{1}{n-2} \mu^2, & \text{if $P \in X$ is of type $cD_n$}, \\
        \frac{1}{6} \mu^2, & \text{if $P \in X$ is of type $cE_6$}, \\
        \frac{1}{12} \mu^2, & \text{if $P \in X$ is of type $cE_7$}, \\
        \frac{1}{30} \mu^2, & \text{if $P \in X$ is of type $cE_8$}.
    \end{dcases}
    \]
\end{theorem}

We note that these inequalities are optimal which are explained in Remarks \ref{optimalcD} and \ref{optimalcE} where we estimate these lower bounds by illustrating many divisorial extractions of minimal discrepancies.  We show that the inequality in Theorem \ref{cAnineq} is optimal as well in Remark \ref{optimalcA}. As an application of Theorem \ref{intro3foldlocalineq}, we prove the birationally rigid for Fano $3$-fold weighted hypersurfaces of index $1$ containing terminal cyclic quotient singularities and $cD_n$ singularities as follows:

\begin{theorem}[=Theorem \ref{rigidity}] \label{introrigidity}
    Let $X$ be a Fano threefold weighted hypersurface of index $1$ listed in Table~\ref{table:Fanohyp} that is quasismooth along the non-Gorenstein locus of the ambient weighted projective $4$-space.
    Let $n'$ be the positive integer in the 5th column ``$cD_n$" of Table~\ref{table:Fanohyp}.
    If $X$ has at most terminal singularities of type $cD_n$ for $n \leq n'$ in addition to terminal cyclic quotient singularities, then $X$ is birationally rigid.
\end{theorem}

In Remark \ref{existcD4}, we show the existence of a Fano $3$-fold in Theorem \ref{introrigidity}, which provides a first example of birationally rigid Fano 3-fold admitting a terminal singular point of type $cD_4$.

%Note that the approach to prove \cite[Theorem 1.2]{KOPP} is to consider a tower of blowups and then to track the local inequality of Fulton–Lazarsfeld in \cite[Theorem 1.3]{FL82} at each level of the tower. In \cite{KOP}, the authors generalized the local inequality of Fulton–Lazarsfeld to a single weighted blow up to obtain Theorem \ref{cAnineq}.

 Note that the method of proving \cite[Theorem 1.2]{KOPP} used a tower of blow-ups and applied the Fulton-Lazarsfeld local inequality \cite[Theorem 1.3]{FL82} step-by-step. In \cite{KOP}, the authors generalized the inequality to a single weighted blow-up, obtaining Theorem \ref{cAnineq}.
In what follows, we explain our argument for proving Theorem \ref{intro3foldlocalineq}. By taking a general hyperplane section $S$ through $P\in X$ and applying the inversion of adjunction, we reduce the inequalities in Theorem \ref{intro3foldlocalineq} to establishing the inequalities for surface $S$ in Theorem \ref{ineqklt}, which follows from a finite cover trick \cite[Proposition 5.20]{KM} and the $4\mu^2$ inequality for smooth surface \cite[Theorem~3.1]{Corti}. Notice that Theorem \ref{intro3foldlocalineq} follows from Theorems \ref{3foldlocalineq} and \ref{ineqDE} as well. In the Appendix, we provide the original and computational approach of Theorem \ref{ineqDE}. 
\begin{remark}\label{normalized local volume}
Note that the coefficients of the inequalities in Theorem \ref{ineqklt} are equal to those for the normalized local volume of the corresponding singularities in dimension $2$ (cf. \cite[Proposition 4.10]{LL19},\cite[Example 7.1]{LX20}).
\end{remark}
\begin{Acknowledgments}
 We would like to thank Kento Fujita and Mircea Musta{\c{t}}{\u{a}} for helpful comments on Theorem \ref{ineqDE} so that we can improve and find an easier argument of Theorem \ref{ineqklt} (cf. Remark \ref{normalized local volume}).
We also would like to thank Akiyoshi Sannai and AI Math Referee, an AI-assisted pipeline developed by him for auditing mathematical papers, for pointing out errors in Proposition~\ref{Prop:compmult} and Lemma~\ref{Lem:quadformE} of an earlier version of the paper.
The first author thanks Jungkai Alfred Chen, Yoshinori Gongyo, Xiaolong Hu for their discussion with him, and thanks the National Taiwan University and NCTS for the warm hospitality during his visit. The second author was partially supported by NCTS and NSTC of Taiwan (Grant Numbers: 114-2115-M-008-004-MY2 and 113-2123-M-002-019). He would like to thank Kyushu University for
its hospitality during his visit. The third author was partially supported by JSPS KAKENHI Grant Number 23K22389.
\end{Acknowledgments}

%%%%%%%%%%%%%%%%%%%%%%%%%%%%%%%%%%
%%%%%%%%%%%%%%%%%%%%%%%%%%%%%%%%%%
 %\section{Preliminaries}
%%%%%%%%%%%%%%%%%%%%%%%%%%%%%%%%%%
%%%%%%%%%%%%%%%%%%%%%%%%%%%%%%%%%%

%%%%%%%%%%%%%%%%%%%%%%%%%%%%%%%%%%
%%%%%%%%%%%%%%%%%%%%%%%%%%%%%%%%%%
\section{Local inequalities for $cD$ and $cE$ singularities}
%%%%%%%%%%%%%%%%%%%%%%%%%%%%%%%%%%
%%%%%%%%%%%%%%%%%%%%%%%%%%%%%%%%%%

Let $P \in S$ be a germ of a normal surface.
For a mobile linear system $\mathcal{L}$ on $S$, denote by $\mathcal{L}^2$ the local intersection multiplicity of two general members of $\mathcal{L}$, that is, $\mathcal{L}^2 = \mult_P (D_1 \cdot D_2)$, where $D_1, D_2 \in \mathcal{L}$ are general members. We begin by establishing a lower bound for $\mathcal{L}^2$.

\begin{theorem} \label{ineqklt}
    Let $P\in S\simeq (o\in \mathbb{C}^2/G)$ be a germ of klt surface singularity 
    where $G\subset GL(2,\mathbb{C})$ is a finite group. Assume that $\mathcal{L}$ is a mobile linear system on $S$ and $\mu$ is a positive rational number such that the pair $(S, \frac{1}{\mu} \mathcal{L})$ is not log canonical at $P$.
    Then one has 
\[ \mathcal{L}^2>\frac{4}{|G|}\mu^2,\]  where $|G|$ is the order of $G$.
    In particular, if $P\in S$ is a Du Val singularity, then 
    \[
    \mathcal{L}^2 > \begin{dcases}
        \frac{4}{n+1} \mu^2, & \text{if $P \in S$ is of type $A_n$}, \\
        \frac{1}{n-2} \mu^2, & \text{if $P \in S$ is of type $D_n$}, \\
        \frac{1}{6} \mu^2, & \text{if $P \in S$ is of type $E_6$}, \\
        \frac{1}{12} \mu^2, & \text{if $P \in S$ is of type $E_7$}, \\
        \frac{1}{30} \mu^2, & \text{if $P \in S$ is of type $E_8$}.
    \end{dcases}
    \]
\end{theorem}
\begin{proof}
%Suppose that the germ $P\in S$ is klt surface singularity. By shrinking if necessary, we may assume that \[(P\in S)\simeq (o\in \mathbb{C}^2/G), \] for some finite group $G\subset GL(2,\mathbb{C})$. 
Denote by $\pi:\mathbb{C}^2\to \mathbb{C}^2/G$ the natural quotient map. Note that the normal surface $S$ is $\mathbb{Q}$-Gorenstein and $\pi$ is \'etale in codimension 2, which implies $K_{\mathbb{C}^2}=\pi^* K_S$. As $(S,\frac{1}{\mu}\mathcal{L})$ is not log canonical at $P$, it follows from \cite[ (4) of Proposition 5.20]{KM} that $(\mathbb{C}^2,\frac{1}{\mu}\pi^*{\mathcal{L}})$ is not log canonical at the origin $O=\pi^{-1}(o)$ of $\mathbb{C}^2$. Thus, for general members $D_1, D_2\in \mathcal{L}$, $$\mathcal{L}^2=\mult_P(D_1\cdot D_2)=\frac{1}{\deg \pi}\mult_O(\pi^*(D_1)\cdot \pi^*(D_2))>\frac{4}{|G|}\mu^2,$$ where the last inequality is established by Corti in \cite[Theorem~3.1]{Corti}. 

In particular, if $P\in S$ is Du Val of type $A_n$  (resp. $D_n, E_6, E_7$ and $E_8$), the above group $G$ is isomorphic to the cyclic group (resp. binary dihedral group binary tetrahedral group,  binary octahedral group, and binary icosahedral group) of order $|G|=n+1$ (resp. $4(n-2), 24, 48$, and $120$). We are done.
\end{proof}

 By using Theorem \ref{ineqklt}, we can generalize Theorem \ref{cAnineq} to $cD$ and $cE$ points.

\begin{theorem} \label{3foldlocalineq}
   Let $P \in X$  be a germ of a 3-dimensional terminal singularity of type either $cD_n$ or $cE_n$, $\mathcal{M}$ a mobile linear system of Cartier divisors on $X$ and $\mu$ a positive rational number.
    Suppose that $P$ is a center of non-canonical singularities of the pair $(X, \frac{1}{\mu} \mathcal{M})$.
    Then, for general members $D_1, D_2 \in \mathcal{M}$, we have
    \[
    \mult_P (D_1 \cdot D_2) >
    \begin{dcases}
        \frac{1}{n-2} \mu^2, & \text{if $P \in X$ is of type $cD_n$}, \\
        \frac{1}{6} \mu^2, & \text{if $P \in X$ is of type $cE_6$}, \\
        \frac{1}{12} \mu^2, & \text{if $P \in X$ is of type $cE_7$}, \\
        \frac{1}{30} \mu^2, & \text{if $P \in X$ is of type $cE_8$}.
    \end{dcases}
    \]
\end{theorem}

\begin{proof}
    Let $S \subset X$ be a general hyperplane section through $P$ so that $P \in S$ is Du Val of type $D_n$ (resp.\ $E_n$) if $P \in X$ is of type $cD_n$ (resp.\ $cE_n$), and that $S$ does not contain any base curve of $\mathcal{M}$.
    Then, $\mathcal{L} := \mathcal{M}|_S$ is a mobile linear system on $S$ and the pair $(X, \frac{1}{\mu} \mathcal{M} + S)$ is not log canonical at $P \in X$.
    By inversion of adjunction, the pair $(S, \frac{1}{\mu} \mathcal{L})$ is not log canonical at $P$.
    Let $D_1, D_2 \in \mathcal{M}$ be general members.
    Since $S$ is a general hyperplane section through $P$, we have
    \[
    \mult_P (D_1 \cdot D_2) = \mult_P (D_1|_S \cdot D_2|_S) = \mathcal{L}^2
    \]
    and the assertion follows from Theorem~\ref{ineqklt}.
\end{proof}

In the rest of this section, we shall show that all inequalities in Theorems \ref{cAnineq} and  \ref{3foldlocalineq} are optimal. At first, we prove the inequality in Theorem \ref{cAnineq} is optimal, by the following example.

\begin{remark}\label{optimalcA}
    We explain that the local inequality for terminal $cA_n$ singularity given in \cite[Theorem~1.3]{KOP} as well as the $4 \mu^2$-inequality for smooth points are optimal.

    Let $n \geq 0$ be an integer and let $P \in X$ be the germ at origin of the hypersurface
    \[
    0 \in \{xy + z^{n+1} + t^{2(n+1)} = 0\} \subset \mathbb{A}^4,
    \]
    which is a germ of a terminal $cA_n$ singularity.
    By an abuse of notation, we regard a $cA_0$ singularity as a smooth point.
    Let $\mathcal{M}$ be the mobile linear system on $X$ generated by $x$ and $y$.
    For general members $D_1, D_2 \in \mathcal{M}$, we have 
    \[
    \mult_P (D_1 \cdot D_2) = n+1.
    \]
    Let $\varphi \colon Y \to X$ be the weighted blow-up of $X$ at $P$ with weight $\mathrm{wt} (x, y, z, t) = (n+1, n+1, 2, 1)$.
    Then $\varphi$ is a birational morphism whose exceptional divisor $E$ is a prime divisor of discrepancy $2$ and we have $\mult_E (\mathcal{M}) = n+1$. Note that $\varphi$ is a divisorial contraction when $n$ is even but not when $n$ is odd since in that case $Y$ has a singularity worse than terminal.
    Hence, for any sufficiently small $\varepsilon > 0$, the point $P$ is the center of non-canonical singularities of the pair $(X, \frac{1}{\mu} \mathcal{M})$, where $\mu := (n+1)/2 - \varepsilon$.
    We can write
    \[
    \mult_P (D_1 \cdot D_2) = n+1 = N_{\varepsilon} \mu^2,
    \]
    where
    \[
    N_{\varepsilon} = \frac{n+1}{((n+1)/2 - \varepsilon)^2} \to \frac{4}{n+1} \quad (\varepsilon \to 0).
    \]
    This shows that the local inequality for $cA_n$ singularity is optimal.
    The above argument for $n = 0$ shows that the $4 \mu^2$-inequality for smooth points is optimal.
\end{remark}

%In what follows, we shall show that  the inequalities in Theorem \ref{3foldlocalineq} and Theorem \ref{ineqklt} are all optimal. 
The following result is useful when we compute the multiplicity of the intersection subvariety of two general members of a mobile linear system on a hypersurface germ.

\begin{proposition} \label{Prop:compmult}
    Let $P \in X$ be the germ at origin of the hypersurface 
    \[
    0 \in \{f (x, y, z, t) = 0\} \subset \mathbb{A}^4,
    \]
    where $f (x, y, z, t) \in \mathbb{C} [x, y, z, t]$.
    We assume that $P \in X$ is an isolated singularity.
    Let $\mathcal{M}$ be the linear system generated by a set $\Lambda \subset \mathbb{C} [x, y, z, t]$ of polynomials.
    Let $a, b, c$ and $d$ be positive integers such that $\gcd \{a, b, c, d\} = 1$.
    We assume that the following conditions are satisfied:
    \begin{enumerate}
        \item The common zero locus of polynomials in $\Lambda$ is the origin $\{P\}$.
        \item Any polynomial in $\Lambda$ is quasi
        homogeneous of degree $m$ with respect to the weight $\mathrm{wt} (x, y, z, t) = (a, b, c, d)$.
        \item Either at least two of $a, b, c, d$ attain their minimum $\min \{a, b, c, d\}$ or there is a unique integer among $a, b, c, d$ that attains their minimum $\min \{a, b, c, d\}$ and the least weight part of $f$ with respect to $\mathrm{wt} (x, y, z, t) = (a, b, c, d)$ is not divisible by the coordinate of the minimum weight.
    \end{enumerate}
    Then, for general members $D_1, D_2 \in \mathcal{M}$, we have
    \[
    \mult_P (D_1 \cdot D_2) = \frac{m^2 \mathrm{wt} (f) \min \{a, b, c, d\}}{abcd},
    \] 
    where $\mathrm{wt} {f}$ is the weight of $f$ with respect to the weights $\mathrm{wt} (x, y, z, t) = (a, b, c, d)$.
\end{proposition}

\begin{proof}
    We set $\mathcal{X} = \mathbb{A}^4$ and let $\mathcal{M}_{\mathcal{X}}$ be the linear system on $\mathcal{X}$ generated by $\Lambda$.
    Let $\Phi \colon \mathcal{Y} \to \mathcal{X}$ be the weighted blow-up of $\mathcal{X}$ at the origin $P$ with weight $\mathrm{wt} (x, y, z, t) = (a, b, c, d)$ and let $\mathcal{E}$ be the exceptional divisor of $\Phi$.
    Let $Y \subset \mathcal{Y}$ be the proper transform of $X$.
    For $i = 1, 2$, let $\mathcal{D}_i$ be a general member of the linear system $\mathcal{M}_{\mathcal{X}}$ on $\mathcal{X}$ and set $D_i := {\mathcal{D}_i}|_X$ which is a general member of $\mathcal{M}$.
    Let $\mathcal{H} \subset \mathcal{X}$ be a general hyperplane of $\mathcal{X}$ through $P$.
    For a divisor $\mathcal{D}$ on $\mathcal{X}$, we denote by $\tilde{\mathcal{D}}$ the proper transform of $\mathcal{D}$ on $\mathcal{Y}$.
    The linear system on $\mathcal{E}$ defined by
    \[
    \tilde{\mathcal{M}}_{\mathcal{X}} |_{\mathcal{E}} := \{\, \tilde{\mathcal{D}}|_{\mathcal{E}} \mid \mathcal{D} \in \mathcal{M}_{\mathcal{X}}\,\}
    \]
    is base point free by the assumptions (1) and (2).
    By the assumption (3), the subscheme
    \[
    Y \cap \tilde{\mathcal{H}} \cap \mathcal{E} \subset \mathcal{E} \cong \mathbb{P} (a, b, c, d)
    \]
    is of $1$-dimensional and hence $\tilde{\mathcal{D}}_1 \cap \tilde{\mathcal{D}}_2 \cap Y \cap \tilde{\mathcal{H}} \cap \mathcal{E} = \emptyset$ since $\tilde{\mathcal{D}}_i|_{\mathcal{E}}$ is a general member of the base point free linear system $\tilde{\mathcal{M}}_{\mathcal{X}} |_{\mathcal{E}}$ for $i = 1, 2$.
    We have $v_{\mathcal{E}} (\mathcal{D}_i) = m$, $v_{\mathcal{E}} (X) = \mathrm{wt} (f)$ and $v_{\mathcal{E}} (\mathcal{H}) = \min \{a, b, c, d\}$, where $v_{\mathcal{E}}$ is the valuation corresponding to the exceptional divisor $\mathcal{E}$.
    By \cite[Theorem~3.1]{KOP}, we obtain
    \[
    \begin{split}
        \mult_P (D_1 \cdot D_2) &= \mult_P (\mathcal{D}_1 \cap \mathcal{D}_2 \cap X \cap \mathcal{H}) \\
        &= \frac{m^2 \mathrm{wt} (f) \min \{a, b, c, d\}}{abcd},
    \end{split}
    \]
    and the proof is complete.
\end{proof}

\begin{remark}\label{optimalcD}
    We explain that the local inequality for terminal $cD_n$ singularity given in Theorem~\ref{3foldlocalineq} is optimal, where $n \geq 4$.
    
    Let $P \in X$ be the germ at origin of the hypersurface
    \[
    0 \in \{x^2 + y^2 t + z^{2(n-1)} + t^{n-1} = 0\} \subset \mathbb{A}^4.
    \]
    
    We claim that $P \in X$ is a terminal $cD_n$ singularity.
    It is clear that $P \in X$ is an isolated singularity.
    Let $P \in S$ be a general hyperplane section of $P \in X$ through $P$ cut by the equation $z = \alpha x + \beta y + \gamma t$, where $\alpha, \beta, \gamma \in \mathbb{C}$ are general.
    Then $S$ is the hypersurface 
    \[
    \{\phi := x^2 + y^2 t + (\alpha x + \beta y + \gamma t)^{2(n-1)} + t^{n-1} = 0\} \subset \mathbb{A}^3.
    \]
    The least weight term of $\phi$ with respect to weight $\mathrm{wt} (x, y, t) = (n-1, n-2, 2)$ is $x^2 + y^2 t + t^{n-1}$.
    By \cite[Corollary~4.7]{Pae}, $P \in S$ is equivalent to the germ
    \[
    0 \in \{x^2 + y^2 t + t^{n-1}\} \subset \mathbb{A}^3,
    \]
    which is Du Val of type $D_n$.
    This proves the claim.
    
    Let $\mathcal{M}$ be the mobile linear system generated by the set of monomials 
    \[
    \Lambda := \{x^{(n-2)}, \, y^{(n-1)}, \, z^{(n-1)(n-2)}, \, t^{(n-1)(n-2)/2}\}.
    \]
    Let $\varphi \colon Y \to X$ be the weighted blow-up of $X$ at the origin $P$ with weight $\mathrm{wt} (x, y, z, t) = (n-1, n-2, 1, 2)$.
    Let $E$ be the exceptional divisor of $\varphi$, which is isomorphic to the weighted hypersurface
    \[
    E \cong \{x^2 + y^2 t + z^{2(n-1)} + t^{n-1} = 0\} \subset \mathbb{P} (n-1, n-2, 1, 2).
    \]
    The variety $Y$ has $3$ singular points among which one is of type $\frac{1}{n-2} (1, 1, n-3)$ and $2$ are of type $\frac{1}{2} (1, 1, 1)$ (all of them are of type $\frac{1}{2} (1, 1, 1)$ when $n = 4$). 
    The singular locus is the set $\{x = z = 0\} \cap E$ (resp.\ $\{y = z = 0\} \cap E \cup \{(0\!:\!1\!:\!0\!:\!0) \in E\}$) if $n$ is even (resp.\ odd).
    It follows that $\varphi$ is a divisorial contraction with center $P \in X$ and it is straightforward to check that its discrepancy is $1$.
    We have $\mult_E (\mathcal{M}) = (n-1)(n-2)$, and hence for a sufficiently small rational number $\varepsilon > 0$, the pair $(X, \frac{1}{\mu} \mathcal{M})$ is not canonical along $E$, where $\mu := (n-1)(n-2) - \varepsilon$.
    Applying Proposition~\ref{Prop:compmult} for the set $\Lambda$ and the weight $\mathrm{wt} (x, y, z, t) = (n-1, n-2, 1, 2)$, we obtain
        \[
    \begin{split}
        \mult_P (D_1 \cdot D_2) &= \mult_P (\mathcal{D}_1 \cap \mathcal{D}_2 \cap X \cap \mathcal{H}) \\
        &= (n-1)^2 (n-2) \\
        &= \frac{(n-1)^2(n-2)}{((n-1)(n-2) - \varepsilon)^2} \mu^2,
    \end{split}
    \]
    and
    \[
    \frac{(n-1)^2(n-2)}{((n-1)(n-2) - \varepsilon)^2} \to \frac{1}{n-2} \quad (\varepsilon \to 0).
    \]
    This shows that the local inequality for terminal $cD_n$ singularity is optimal.
    Similarly, consider the hyperplane section $S=\{z=0\}\cap X\subset \mathbb{A}^4$. Then $(S,\frac{1}{\mu}\mathcal{M}|_S)$ is not log canonical along $E|_S$ and \[\mult_P ({D_1}|_S \cdot {D_2}|_S)=\mult_P (D_1 \cdot D_2).\]
    Therefore, the local inequality for terminal $D_n$ singularity in Theorem \ref{ineqklt} is optimal.
\end{remark}

\begin{remark}\label{optimalcE}
    We explain that the local inequality for terminal $cE_n$ singularity given in Theorem~\ref{3foldlocalineq} is optimal, where $n = 6, 7, 8$.
    Let $P \in X$ be the germ at origin of the hypersurface
    \[
    0 \in \{f = 0\} \subset \mathbb{A}^4,
    \]
    where
    \[
    f =
    \begin{cases}
        x^2 + y^3 + z^4 + t^{12}, & \text{if $n = 6$}, \\
        x^2 + y^3 + y z^3 + t^{18}, & \text{if $n = 7$}, \\
        x^2 + y^3 + z^5 + t^{30}, & \text{if $n = 8$}.
    \end{cases}
    \] 
    
    We claim that $P \in X$ is a terminal $cE_n$ singularity.
    It is clear that $P \in X$ is an isolated singularity.
    We set
    \[
    (a, b, c) := 
    \begin{cases}
        (6, 4, 3), & \text{if $n = 6$}, \\
        (9, 6, 4), & \text{if $n = 7$}, \\
        (15, 10, 6), & \text{if $n = 8$}.
    \end{cases}
    \]
    Let $P \in S$ be a general hyperplane of $P \in X$ through $P$ cut by the equation $t = \alpha x + \beta y + \gamma z$, where $\alpha, \beta, \gamma \in \mathbb{C}$ are general.
    Then $S$ is the hypersurface
    \[
    \{\phi (x, y, z) = 0\} \subset \mathbb{A}^3,
    \]
    where $\phi := f (x, y, z, \alpha x + \beta y + \gamma z)$, and the least weight term of $\phi$ with respect to $\mathrm{wt} (x, y, z) = (a, b, c)$ is the polynomial
    \[
    \begin{cases}
        x^2 + y^3 + z^4, & \text{if $n = 6$}, \\
        x^2 + y^3 + y z^3, & \text{if $n = 7$}, \\
        x^2 + y^3 + z^5, & \text{if $n = 8$}.
    \end{cases}
    \]
    By \cite[Corollary~4.7]{Pae}, $P \in S$ is Du Val of type $E_n$.
    This proves the claim.
    
    Let $(a, b, c)$ be as above and we set $m := \mathrm{lcm} \{ a, b, c\}$.
    Let $\mathcal{M}$ be the mobile linear system on $X$ generated by the set of monomials
    \[
    \Lambda := \{x^{m/a}, \, y^{m/b}, \, z^{m/c}, \, t^m\}.
    \]
    Let $\varphi \colon Y \to X$ be the weighted blow-up of $X$ at $P$ with weight $\mathrm{wt} (x, y, z, t) = (a, b, c, 1)$, which is a divisorial contraction of discrepancy $1$ with center $P \in X$.
    Let $E$ be the exceptional divisor of $\varphi$.
    The variety $Y$ has $3$ singular points:
    \[
    \Sing (Y) =
    \begin{cases}
        \left\{ 2 \times \frac{1}{2} (1, 1, 1), \frac{1}{3} (1, 1, 2) \right\}, & \text{if $n = 6$}, \\
        \left\{ \frac{1}{2} (1, 1, 1), \frac{1}{3} (1, 1, 2), \frac{1}{4} (1, 1, 3) \right\}, & \text{if $n = 7$}, \\
        \left\{\frac{1}{2} (1, 1, 1), \frac{1}{3} (1, 1, 2), \frac{1}{5} (1, 1, 4) \right\}, & \text{if $n = 8$},
    \end{cases}
    \]
    
    It follows that $\varphi$ is a divisorial contraction with center $P \in X$ and it is straightforward to check that its discrepancy is $1$.
    We have $\mult_E (\mathcal{M}) = m$, and hence for a sufficiently small rational number $\varepsilon > 0$, the pair $(X, \frac{1}{\mu} \mathcal{M})$ is not canonical along $E$, where $\mu := m - \varepsilon$.
    Applying Proposition~\ref{Prop:compmult} for the set $\Lambda$ and the weight $\mathrm{wt} (x, y, z, t) = (a, b, c, 1)$, we obtain
    \[
    \mult_P (D_1 \cdot D_2) = \mult_P (\mathcal{D}_1 \cap \mathcal{D}_2 \cap X \cap \mathcal{H}) 
    = \frac{2 m^2}{bc} = N_{\varepsilon} \mu^2,
    \]
    where
    \[
    N_{\varepsilon} = 
    \begin{dcases}
        \frac{1}{6} \cdot \frac{m^2}{(m-\varepsilon)^2} \to \frac{1}{6}, & \text{if $n = 6$}, \\
        \frac{1}{12} \cdot \frac{m^2}{(m-\varepsilon)^2} \to \frac{1}{12}, & \text{if $n = 7$}, \\
        \frac{1}{30} \cdot \frac{m^2}{(m-\varepsilon)^2} \to \frac{1}{30}, & \text{if $n = 8$},
    \end{dcases}
    \]
    as $\varepsilon \to 0$.
    This shows that the local inequality for terminal $cE_n$ singularity is optimal. 
    Similarly, consider the hyperplane section $S=\{t=0\}\cap X\subset \mathbb{A}^4$. Then $(S,\frac{1}{\mu}\mathcal{M}|_S)$ is not log canonical along $E|_S$ and \[\mult_P ({D_1}|_S \cdot {D_2}|_S)=\mult_P (D_1 \cdot D_2).\]
    Therefore, the local inequality for terminal $E_n$ singularity in Theorem \ref{ineqklt} is optimal.
\end{remark}

%%%%%%%%%%%%%%%%%%%%%%%%%%%%%%%%%
%%%%%%%%%%%%%%%%%%%%%%%%%%%%%%%%%
\section{Application to birational rigidity of Fano 3-folds}
%%%%%%%%%%%%%%%%%%%%%%%%%%%%%%%%%
%%%%%%%%%%%%%%%%%%%%%%%%%%%%%%%%%

As an application of the local inequality obtained in the previous section, we prove birational rigidity of certain Fano threefolds admitting terminal singular points of type $cD_n$.

We recall basic definitions regarding weighted projective varieties.
For positive integers $a_0, \dots, a_4$, let
\[
R (a_0, \dots, a_4) = \mathbb{C} [x_0, \dots, x_4]
\]
be the graded ring whose grading is given by $\deg x_i = a_i$ for $0 \leq i \leq 4$.
A weighted projective $4$-space
\[
\mathbb{P} \coloneq \mathbb{P} (a_0, \dots, a_4) = \Proj R (a_0, \dots, a_4),
\]
is \textit{well-formed} if the greatest common divisor of any $4$ of the weights $a_0, \dots, a_4$ is $1$.
Let 
\[
X = \Proj R (a_0, \dots, a_4)/(f) \subset \mathbb{P}
\]
be a weighted hypersurface in a well-formed weighted projective $4$-space $\mathbb{P}$ defined by a homogeneous polynomial $f \in R (a_0, \dots, a_4)$ of degree $d$.
We say that $X$ is a \textit{Fano threefold weighted hypersurface of index $1$} if $X$ is $\mathbb{Q}$-factorial, has only terminal singularities and 
\[
d = \sum_{i=0}^4 a_i - 1.
\]
The \textit{non-quasismooth locus} of $X$, denoted $\mathrm{NQsm} (X)$, is defined by
\[
\mathrm{NQsm} (X) \coloneq \Pi (\Sing (C_X) \setminus \{0\}),
\]
where $\Pi \colon \mathbb{A}^5 \setminus \{0\} \to \mathbb{P}$ is the natural projection and 
\[
C_X \coloneq \Spec\ \mathbb{C} [x_0, \dots, x_4]/(f) \subset \mathbb{A}^5
\]
is the affine cone of $X$.
The \textit{quasismooth locus} of $X$, denoted $\mathrm{Qsm} (X)$, is the complement of $\mathrm{NQsm} (X)$ in $X$.
We say that $X$ is \textit{quasismooth along the non-Gorenstein locus of its ambient space} if $\Sing (\mathbb{P}) \cap X \subset \mathrm{Qsm} (X)$.
This implies that $X$ has cyclic quotient singularities along $\Sing (\mathbb{P}) \cap X$ that are naturally induced from $\mathbb{P}$.

\begin{theorem} \label{rigidity}
    Let $X$ be a Fano threefold weighted hypersurface of index $1$ listed in Table~\ref{table:Fanohyp} that is quasismooth along the non-Gorenstein locus of the ambient weighted projective $4$-space.
    Let $n'$ be the positive integer in the 5th column ``$cD_n$" of Table~\ref{table:Fanohyp}.
    If $X$ has at most terminal singularities of type $cD_n$ for $n \leq n'$ in addition to terminal cyclic quotient singularities, then $X$ is birationally rigid.
\end{theorem}

\begin{remark} \label{lic}
    The number $l_{\mathrm{ic}}$ in Table~\ref{table:Fanohyp} is the one given in \cite[Table~1]{KOP} and the number $n'$ given in the $5$th column of Table~\ref{table:Fanohyp} is defined by
    \[
    n' := \lfloor \frac{1}{l_{\mathrm{ic}} (-K_X)^3} \rfloor + 2.
    \]
\end{remark}

\begin{table}
    \caption{Fano $3$-fold weighted hypersurfaces of index $1$}
    \label{table:Fanohyp}
    \centering
    \begin{tabular}{clccc}
    \toprule
    \textnumero & $X_d \subset \mathbb{P} (a_0,\dots,a_4)$ & $-K_X^3$ & $l_{\mathrm{ic}}$ & $cD_n$ \\
    \midrule
    37 & $X_{18} \subset \mathbb{P} (1,2,3,4,9)$ & $1/12$ & $6$ & $4$ \\
    41 & $X_{20} \subset \mathbb{P} (1,1,4,5,10)$ & $1/10$ & $5$ & $4$ \\
    49 & $X_{21} \subset \mathbb{P} (1,3,5,6,7)$ & $1/30$ & $15$ & $4$ \\
    52 & $X_{22} \subset \mathbb{P} (1,2,4,5,11)$ & $1/20$ & $10$ & $4$ \\
    54 & $X_{24} \subset \mathbb{P} (1,1,6,8,9)$ & $1/18$ & $9$ & $4$ \\
    57 & $X_{24} \subset \mathbb{P} (1,3,4,5,12)$ & $1/30$ & $15$ & $4$ \\
    59 & $X_{24} \subset \mathbb{P} (1,3,6,7,8)$ & $1/42$ & $21$ & $4$ \\
    62 & $X_{26} \subset \mathbb{P} (1,1,5,7,13)$ & $2/35$ & $7$ & $4$ \\
    63 & $X_{26} \subset \mathbb{P} (1,2,3,8,13)$ & $1/24$ & $8$ & $5$ \\
    64 & $X_{26} \subset \mathbb{P} (1,2,5,6,13)$ & $1/30$ & $10$ & $5$ \\
    66 & $X_{27} \subset \mathbb{P} (1,5,6,7,9)$ & $1/70$ & $35$ & $4$ \\
    67 & $X_{28} \subset \mathbb{P} (1,1,4,9,14)$ & $1/18$ & $9$ & $4$ \\
    68 & $X_{28} \subset \mathbb{P} (1,3,4,7,14)$ & $1/42$ & $21$ & $4$ \\
    70 & $X_{30} \subset \mathbb{P} (1,1,4,10,15)$ & $1/20$ & $10$ & $4$ \\
    71 & $X_{30} \subset \mathbb{P} (1,1,6,8,15)$ & $1/24$ & $8$ & $5$ \\
    72 & $X_{30} \subset \mathbb{P} (1,2,3,10,15)$ & $1/30$ & $10$ & $5$ \\
    73 & $X_{30} \subset \mathbb{P} (1,2,6,7,15)$ & $1/42$ & $14$ & $5$ \\
    75 & $X_{30} \subset \mathbb{P} (1,4,5,6,15)$ & $1/60$ & $20$ & $5$ \\ 
    77 & $X_{32} \subset \mathbb{P} (1,2,5,9,16)$ & $1/45$ & $18$ & $4$ \\
    78 & $X_{32} \subset \mathbb{P} (1,4,5,7,16)$ & $1/70$ & $28$ & $4$ \\
    80 & $X_{34} \subset \mathbb{P} (1,3,4,10,17)$ & $1/60$ & $30$ & $4$ \\
    81 & $X_{34} \subset \mathbb{P} (1,4,6,7,17)$ & $1/84$ & $28$ & $5$ \\
    82 & $X_{36} \subset \mathbb{P} (1,1,5,12,18)$ & $1/30$ & $12$ & $4$ \\
    83 & $X_{36} \subset \mathbb{P} (1,3,4,11,18)$ & $1/66$ & $33$ & $4$ \\
    84 & $X_{36} \subset \mathbb{P} (1,7,8,9,12)$ & $1/168$ & $63$ & $4$ \\
    85 & $X_{38} \subset \mathbb{P} (1,3,5,11,19)$ & $2/165$ & $33$ & $4$ \\
    86 & $X_{38} \subset \mathbb{P} (1,5,6,8,19)$ & $1/120$ & $40$ & $5$ \\
    87 & $X_{40} \subset \mathbb{P} (1,5,7,8,20)$ & $1/140$ & $40$ & $5$ \\
    88 & $X_{42} \subset \mathbb{P} (1,1,6,14,21)$ & $1/42$ & $14$ & $5$ \\
    89 & $X_{42} \subset \mathbb{P} (1,2,5,14,21)$ & $1/70$ & $14$ & $7$ \\
    90 & $X_{42} \subset \mathbb{P} (1,3,4,14,21)$ & $1/84$ & $42$ & $4$ \\
    91 & $X_{44} \subset \mathbb{P} (1,4,5,13,22)$ & $1/130$ & $52$ & $4$ \\
    92 & $X_{48} \subset \mathbb{P} (1,3,5,16,24)$ & $1/120$ & $48$ & $4$ \\
    93 & $X_{50} \subset \mathbb{P} (1,7,8,10,25)$ & $1/280$ & $70$ & $6$ \\
    94 & $X_{54} \subset \mathbb{P} (1,4,5,18,27)$ & $1/180$ & $36$ & $7$ \\
    95 & $X_{66} \subset \mathbb{P} (1,5,6,22,33)$ & $1/330$ & $110$ & $5$ \\
    \bottomrule
    \end{tabular}
\end{table}

Before giving the proof, we recall the notion of isolating class.

\begin{definition}
    Let $P \in X$ be a point contained in the smooth locus of $\mathbb{P}$ and let $L$ be a divisor class on $X$.
    For a positive integer $k$, we define $|\mathfrak{m}_P^k (kL)|$ to be the sublinear system of $|k L|$ consisting of divisors vanishing at $P$ with order at least $k$.
    We say that $L$ \textit{isolates} $P$ or $L$ is a \textit{$P$-isolating class} if $P$ is an isolated component of the base locus of $|\mathfrak{m}_P^k (k L)|$ for some $k > 0$.

\end{definition}

\begin{proof}[Proof of Theorem~\ref{rigidity}]
    Let $X$ be as in Theorem~\ref{rigidity}.
    By \cite[Propositions~5.5 and 5.12]{KOPP}, no curve and no smooth point on $X$ is a maximal center.
    Moreover, by \cite[Proposition~5.15]{KOPP}, for each terminal quotient singular point $P \in X$, either $P$ is not a maximal center or there exists an elementary-self link initiated by the Kawamata blow-up $Y \to X$ of $X$ at $P$.

    Let $P \in X$ be a terminal singular point of type $cD_n$, where $n \leq n'$, and let $l \coloneq l_{\mathrm{ic}}$ be the positive integer in the 4th column of Table~\ref{table:Fanohyp}.
    Suppose that $P \in X$ is a maximal center of a mobile linear system $\mathcal{M} \sim - \mu K_X$, where $\mu \in \mathbb{Q}_{> 0}$.
    Let $D_1, D_2 \in \mathcal{M}$ be general members.
    By Theorem~\ref{3foldlocalineq}, we have
    \[
    \mult_P (D_1 \cdot D_2) > \frac{1}{n-2} \mu^2.
    \]
    By \cite[Lemmas~6.9 and 6.10]{KOP}, there is a $P$-isolating class $-l K_X$, where $l \leq l_{\mathrm{ic}}$.
    It follows that there is an effective $\mathbb{Q}$-divisor $T \sim_{\mathbb{Q}} - l K_X$ such that
    \[
    \ord_P (T) \coloneq \sup \{\, j \geq 0 \mid f_T \in \mathfrak{m}_{X, P}^{j}\,\} \geq 1, 
    \]
    where $f_T \in \mathcal{O}_{X, P}$ is the local equation of $T$ at $P$, and that $\Supp (T)$ does not contain any irreducible component of $D_1 \cap D_2$.
    By \cite[Lemma~2.11]{KOP}, we have
    \[
    l \mu^2 (-K_X)^3 = (T \cdot D_1 \cdot D_2) \geq \ord_P (T) \mult_P (D_1 \cdot D_2) > \frac{1}{n-2} \mu^2,
    \]
    which implies
    \[
    n'-2 \geq n-2 > \frac{1}{l (-K_X)^3} \geq \frac{1}{l_{\mathrm{ic}} (-K_X)^3}.
    \]
    This is impossible by the definition of $n'$ (see Remark~\ref{lic}).
    This shows that $P \in X$ is not a maximal center.

    By \cite[Theorem~2.14]{KOP}, $X$ is birationally rigid.
\end{proof}

\begin{remark} \label{existcD4}
    We explain the existence of a birationally rigid Fano 3-fold weighted hypersurface of index $1$ with a singular point of type $cD_4$.
    Let $X = X_{18} \subset \mathbb{P} := \mathbb{P} (1, 2, 3, 4, 9)$ be a general member of the linear system $\Sigma \subset |\mathcal{O}_{\mathbb{P}} (18)|$ on $\mathbb{P}$ generated by the monomials
    \[
    w^2, \, t^2 z x^7, \, y^3 x^{12}, \, z^3 x^9, \, t^3 y, \, z^6, \, y^9, \, t^4 x^2.
    \]
    Note that $X$ belongs to family \textnumero~37.
    The base locus of $\Sigma$ is $\{P, Q\}$, where $P := (1:0:0:0:0)$ and $Q := (0:0:0:1:0)$.
    It is easily verified that $X$ is quasismooth at $Q$, and hence $P$ is the unique non-quasismooth point of $X$.
    
    We claim that $P \in X$ is a terminal singularity of type $cD_4$.
    Rescaling coordinates, we may assume that
    \[
    f = w^2 + t^2 z x^7 + y^3 x^{12} + z^3 x^9 + g,
    \]
    where $g = g (x, y, z, t)$ is a linear combination of the monomials $t^3 y, z^6, y^9$ and $t^4 x^2$ so that $g \in (y, z, t)^4$.
    By setting $x = 1$, the germ $P \in X$ is isomorphic to
    \[
    0 \in \{\, w^2 + t^2 z + y^3 + z^3 + h = 0\,\} \subset \mathbb{A}^4,
    \]
    where $h := g (1,y,z,t) \in (y,z,t)^4$.
    Let $P \in S$ be the general hyperplane section of $P \in X$ cut by an equation $y = \alpha z + \beta t + \gamma w$, where $\alpha, \beta, \gamma \in \mathbb{C}$ are general.
    The germ $P \in S$ is isomorphic to the germ
    \[
    0 \in \{\, \phi (z, t, w) = 0 \,\} \subset \mathbb{A}^3,
    \]
    where
    \[
    \phi := w^2 + t^2 z + (\alpha z + \beta t + \gamma w)^3 + z^3 + h (\alpha z + \beta t + \gamma w, z, t) = 0
    \]
    The least weight terms of $\phi$ with respect to the weight $\mathrm{wt} (z, t, w) = (2,2,3)$ is $w^2 + t^2 z + (\alpha z + \beta t)^3 + z^3$.
    By \cite[Corollary~4.7]{Pae}, $P \in S$ is equivalent to the germ
    \[
    0 \in \{\, w^2 + t^2 z + (\alpha z + \beta t)^3 + z^3 = 0 \,\} \subset \mathbb{A}^3.
    \]
    It is then easy to see that $P \in S$ is a Du Val singularity of type $D_4$, and hence $P \in X$ is of type $cD_4$.

    The crucial part is to show that $X$ is $\mathbb{Q}$-factorial.
    Let $\Phi \colon \mathcal{Y} \to \mathbb{P}$ be the weighted blowup of $\mathbb{P}$ at the $cD_4$ point $P$ with weight $\mathrm{wt} (y,z,t,w) = (1,1,1,2)$ and let $\mathcal{E} \cong \mathbb{P} (1,1,1,2)$ be the exceptional divisor of $\Phi$.
    We set $Y := \Phi_*^{-1} X$ and let $\varphi := \Phi|_Y \colon Y \to X$ be the induced birational morphism.
    By slight abuse of notation, we identify $\mathcal{E}$ with the weighted projective space $\mathbb{P} (1,1,1,2)$ with homogeneous coordinates $y, z, t, w$ of weights $1,1,1,2$, respectively.
    The restriction $E := \mathcal{E}|_Y$ is isomorphic to  the hypersurface
    \[
    E \cong \{\, t^2 z + y^3 + z^3 = 0\,\} \subset \mathbb{P} (1,1,1,2)
    \]
    and it is smooth outside the point $\tilde{P} := (0:0:0:1) \in \mathcal{E}$.
    This implies that $Y$ is smooth outside $\tilde{P} \in Y$.
    It is also easy to check that the singularity $\tilde{P} \in Y$ is of type $\frac{1}{2} (1,1,1)$.
    This in particular shows that $Y$ is $\mathbb{Q}$-factorial.
    As a divisor on $\mathcal{Y}$, we have $Y = \Phi^*X - 3 \mathcal{E} \sim 18 \Phi^*\mathcal{H} - 3 \mathcal{E}$, where $\mathcal{H} \in \operatorname{Cl} (\mathbb{P})$ is the class such that $\mathcal{O}_{\mathbb{P}} (\mathcal{H}) \cong \mathcal{O}_{\mathbb{P}} (1)$.
    \textcolor{black}{For a variable $s \in \{x, y, z, t, w\}$, we set $\mathcal{D}_s := \{s = 0\} \subset \mathbb{P}$ and $\tilde{\mathcal{D}}_s := \Phi_*^{-1} \mathcal{D}_s$.
    We have $\tilde{\mathcal{D}}_x \sim \Phi^*\mathcal{H}$, $\tilde{\mathcal{D}}_y \sim 2 \Phi^*\mathcal{H} - \mathcal{E}$, $\tilde{\mathcal{D}}_z \sim 3 \Phi^*\mathcal{H} - \mathcal{E}$, $\tilde{\mathcal{D}}_t \sim 4 \Phi^*\mathcal{H} - \mathcal{E}$ and $\tilde{\mathcal{D}}_w \sim 9\Phi^*\mathcal{H} - 2 \mathcal{E}$.
    We see that
    \[
    18 \tilde{\mathcal{D}}_y + 10 \mathcal{E}, 12 \tilde{\mathcal{D}}_z + 4 \mathcal{E}, 9 \tilde{\mathcal{D}}_t + \mathcal{E}, 4 \tilde{\mathcal{D}}_w, 8 \tilde{\mathcal{D}}_y + 20 \tilde{\mathcal{D}}_x, 8 \tilde{\mathcal{D}}_z + 12 \tilde{\mathcal{D}}_x, 8 \tilde{\mathcal{D}}_t + 4 \tilde{\mathcal{D}}_x
    \]
    are members of the linear system $|4(9 \Phi^*\mathcal{H} - 2 \mathcal{E})|$ and the intersection of their supports is empty since $\tilde{\mathcal{D}}_y \cap \tilde{\mathcal{D}}_z \cap \tilde{\mathcal{D}}_t \cap \tilde{\mathcal{D}}_w = \tilde{\mathcal{D}}_x \cap \mathcal{E} = \emptyset$.
    This shows that $9\Phi^*\mathcal{H}-2\mathcal{E}$ is semiample.
    Since the Picard number of $\mathcal{Y}$ is $2$, the cone of nef divisors on $\mathcal{Y}$ is spanned by $2$ rays among which one is the ray generated by $\Phi^*\mathcal{H}$ and the other ray is generated by $\Phi^*\mathcal{H} - r \mathcal{E}$ for some $r \geq 2/9$. It follows that $Y \sim 18 \Phi^*\mathcal{H} - 3 \mathcal{E}$ is in the interior of the nef cone, hence it is ample.}
    %\textcolor{blue}{CHEN: I am sorry for not being familiar with the computations. I don't know why  $9\Phi^*\mathcal{H}-2\mathcal{E}$ is semiample and $Y \sim 18 \Phi^*\mathcal{H} - 3 \mathcal{E}$ is ample. Are there any references to see the details?}
    By the above arguments, we see that $Y$ is quasismooth, that is, its affine quasi-cone is smooth outside the irrelevant locus of the rank $2$ simplicial toric variety $\mathcal{Y}$ (see \cite[Definition~3.1]{BC}).
    By \cite[Proposition~10.8]{BC}, we see that the Picard number of $Y$ is $2$.
    It follows that $\varphi$ is a $K_Y$-negative extremal divisorial contraction from a normal projective $\mathbb{Q}$-factorial variety with only terminal singularities and this implies that $X$ is $\mathbb{Q}$-factorial and its Picard number is $1$.
    By Theorem~\ref{rigidity}, $X$ is birationally rigid.
    \end{remark}

\section{Appendix. Direct computations of generalized $4\mu^2$ formulas to Du Val singularities of type-$D$ and type-$E$.}

%Recall that Theorem \ref{ineqklt} is a generalization of 
The aim of the Appendix is to provide a computational approach, to prove Theorem \ref{ineqDE}, which establishes the lower bounds of the self-intersection $\mathcal{L}^2$ for Du Val singularities of type-$D$ or type-$E$ for mobile linear system $\mathcal{L}$ of $S$ in Theorem \ref{ineqklt}. The idea of obtaining these lower bounds of $\mathcal{L}^2$ is to calculate the positive semi-definiteness of the associated quadratic forms in Lemmas \ref{Lem:quadformD} and \ref{Lem:quadformE} to the minimal resolution  $\pi\colon T\to S$ and adopting the generalized inequality $4\mu^2$ in \cite[Theorem 3.1]{Corti} on $T$.

%We begin with the basic computations on the following positive semi-definiteness of the associated quadratic forms.

\begin{lemma} \label{Lem:quadformD}
    Let $n$ and $k$ be integers such that $n \geq 4$ and $1 \leq k \leq n$.
    We set
    \[
    q_{D_n} = \sum_{i=1}^{n} x_i^2 - \sum_{i=1}^{n-2} x_i x_{i+1} - x_{n-2} x_n.
    \]
    Then the following assertions hold.
    \begin{enumerate}
        \item The quadratic form 
        \[
        q_{D_n} - \frac{1}{2(n-2)} x_k^2
        \] 
        is positive semidefinite for any $k$.
        \item The quadratic form
        \[
        q_{D_n} - \frac{1}{2(n-3)} x_k^2
        \] 
        is positive semidefinite for any $k \ne n-2$.
    \end{enumerate}
\end{lemma}

\begin{proof}
    Let $c$ be a real number and let $M^{(k)}$ be the $n \times n$ symmetric matrix associated with the quadratic form $2 (q_{D_n} - c x_k^2)$.
    For $1 \leq i \leq n$, let $M^{(k)}_i$ be the $i \times i$ submatrix of $M^{(k)}$ consisting of the first $i$ rows and columns of $M^{(k)}$ and set $m^{(k)}_i := \det M^{(k)}_i$.
    Clearly we have $m_1^{(k)} = 2$ if $k \geq 2$ and $m_2^{(k)} = 3$ if $k \geq 3$.
    By the cofactor expansion of $M^{(k)}_i$ along the $i$th row, we obtain the following relations: 
    \begin{itemize}
        \item $m_i^{(k)} = 2 m_{i-1}^{(k)} - m^{(k)}_{i-2}$ for $i \neq k, n$.
        \item $m_k^{(k)} = (2-2c) m^{(k)}_{k-1} - m^{(k)}_{k-2}$ for $2 \leq k \leq n-1$.
        \item $m_n^{(k)} = 
        \begin{cases}
            2 m_{n-1}^{(k)} - 2 m_{n-3}^{(k)}, & \text{if $1 \leq k \leq n-2$}, \\
            %2 m_{n-1}^{(k)} - 2 m_{n-3}^{(n)}, & \text{if $k = n-2$}, \\
            2 m_{n-1}^{(n-1)} - (2-2c) m_{n-3}^{(n-1)}, & \text{if $k = n-1$}, \\
            (2 - 2c) m_{n-1}^{(n)} - 2 m_{n-3}^{(n)}, & \text{if $k = n$}.
        \end{cases}$
    \end{itemize}
    
    From the above relations, we obtain
    \[
    m^{(k)}_i =
    \begin{cases}
        i+1, & \text{if $1 \leq i \leq k-1$}, \\
        (i+1) - 2 (i-k+1) k c, & \text{if $k \leq i \leq n-1$},
    \end{cases}
    \]
    and
    \[
    m^{(k)}_n =
    \begin{cases}
    4 - 8 k c, & \text{if $1 \leq k \leq n-2$}, \\
    4 - 2 n c, & \text{if $k = n-1, n$}.
    \end{cases}
    \]
    
    Now we think of $c$ as a variable and we regard $m^{(k)}_i$ as a linear function on $c$. 
    We then define $\lambda^{(k)}_i$ to be the root of the linear function $m^{(k)}_i$ with respect to the variable $c$ (if $m^{(k)}_i$ does not involve the variable $c$, then we set $\lambda^{(k)}_i = \infty$).
    Then it is easy to see that
    \[
    \min \{\, \lambda^{(k)}_i \mid 1 \leq i, k \leq n\} = \lambda_n^{(n-2)} = \frac{1}{2(n-2)}
    \]
    and
    \[
    \min \{\, \lambda^{(k)}_i \mid 1 \leq i, k \leq n, k \neq n-2 \,\} = \lambda^{(n-3)}_n = \frac{1}{2(n-3)}.
    \]
    These prove the assertions (1) and (2).
\end{proof}

Similarly to Lemma \ref{Lem:quadformD}, we have a positive semidefinite for Du Val singularity of type-$E$.

\begin{lemma} \label{Lem:quadformE}
    Let $n$ and $k$ be integers such that $n\in \{6,7,8\}$ and $1 \leq k \leq n$.
    We set
    \[
    q_{E_n} := \sum_{i=1}^n x_i^2 - \sum_{i=1}^{n-2} x_i x_{i+1} - x_3 x_n.
    \] 
    Then the following assertions hold.
    \begin{enumerate}
        \item The quadratic form 
        \[
        q_{E_n} - \frac{9-n}{12(n-3)} x_k^2
        \] 
        is positive semidefinite for any $k$.
        \item The quadratic form
        \[
        q_{E_n} - \frac{9-n}{10(n-4)} x_k^2
        \]
        is positive semidefinite for any $k \ne 3$.
    \end{enumerate}
\end{lemma}

\begin{proof}
    Let $c$ be a real number and let $M^{(k)}$ be the symmetric $n \times n$ matrix associated with the quadratic form $2(q_{E_n} - c x_k^2)$. 
    For $1 \leq i \leq n$, we denote by $m^{(k)}_i$ the determinant of the $i$th submatric of $M$ consisting of the first $i$ rows and columns.
    We have the following relations: 
    \begin{itemize}
        \item $m^{(k)}_1 = 2$ if $k \geq 2$ and $m_2 = 3$ if $k \geq 3$.
        \item $m^{(k)}_i = 2 m^{(k)}_{i-1} - m^{(k)}_{i-2}$ for $i \neq k, n$;
        \item $m^{(k)}_k = 
        \begin{cases}
            2-2c, & \text{if $k = 1$}, \\
            3-4c, & \text{if $k = 2$}, \\
            (2-2c) m^{(k)}_{k-1} - m^{(k)}_{k-2} = k+1-2kc, & \text{if $3 \leq k \leq n-1$}.
        \end{cases}$
        \item $m_n^{(k)} = 
        \begin{cases}
            2 m^{(k)}_{n-1} - (n-3)(3-4c), & \text{if $k=1, 2$}, \\
            2 m^{(k)}_{n-1} - 3 (n-3), & \text{if $k = 3$}, \\
            2 m^{(k)}_{n-1} - 3 m^{(k-3)}_{n-4}, & \text{if $4 \leq k \leq n-1$}, \\
        \end{cases}$
    \end{itemize}
    %(\textcolor{red}{OKADA: I confirmed that your computation is correct! Thank you very much.}

    %\noindent{\color{blue}(CHEN: The equality $m_{n}^{(k)}= 2 m^{(k)}_{n-1} - 3 m^{(k-3)}_{n-4}$ seems correct. For example, when $n=6, k=4$, the matrix $M^{(4)}$ is $ \begin{bmatrix} 2 & -1 & 0 & 0 & 0 & 0 \\ -1 & 2 & -1 & 0 & 0 & 0 \\0 & -1 & 2 & -1 & 0 & -1 \\0 & 0 & -1 & 2-2c & -1 & 0 \\0 & 0 & 0 & -1 & 2 & 0 \\0 & 0 & -1 & 0 & 0 & 2 \end{bmatrix}$, we have $m_n^{(k)}=2 m^{(k)}_{n-1} - 3 m^{(k-3)}_{n-4}=2 m^{(4)}_{5} - 3 m^{(1)}_{2}=3 -20c$ while $2m_{n-1}^{(k)}-m_{n-2}^{(k-1)}=2m_5^{4}-m_4^3=7-20c$.)} 
    
    By the above relations, we have
    \[
    m^{(k)}_i =
    \begin{cases}
        i+1, & \text{for $1 \leq i \leq k-1$}, \\
        %k+1 - k c, & \text{for $i = k$}, \\
        %k+2 - 2 k c, & \text{for $i = k + 1$}, \\
        %k+3-3kc, & \text{for $i = k+2$}, \\
        (i+1) - 2(i-k+1)kc, & \text{for $k \leq i \leq n-1$}.
    \end{cases}
    \]
    By direct computation, we have
    \[
    m^{(k)}_n =
    \begin{cases}
        (9-n) - 8c, & \text{if $k = 1$}, \\
        (9-n) - 4 (n-1)c, & \text{if $k = 2$}, \\
%        (9-n) - 12(n-3)c, & \text{if $k = 3$}, \\
        (9-n) - 2 (n-k)(9-k)c, & \text{if $ 3 \leq k \leq n-1$}, \\
        (9-n) - 2 n c, & \text{if $k = n$}.
    \end{cases}
    \]
    Now we think of $c$ as a variable and regard $m^{(k)}_i$ as a function on the variable $c$.
    We then define $\lambda^{(k)}_i$ to be the root of the linear function $m^{(k)}_i$ (if $m^{(k)}_i$ does not involve $c$, then we set $\lambda^{(k)}_i = \infty$).
    Then it is easy to see that
    \[
    \min \{\, \lambda^{(k)}_i \mid 1 \leq i, k \leq n\} = \lambda_n^{(3)} = \frac{9-n}{12(n-3)}
    \]
    and
    \[
    \min \{\, \lambda^{(k)}_i \mid 1 \leq i, k \leq n, k \neq 3 \,\} = \lambda^{(4)}_n = {\frac{9-n}{10(n-4)}}.
    \]
    These prove the assertions (1) and (2).
\end{proof}

We also need the following basic estimate.

\begin{lemma} \label{Lem:minfunc}
Let $\lambda, \mu > 0$ be positive numbers with \[ \lambda + 2 \sqrt{\lambda/\mu} < 2, \textup{ and } \mu<2.\] Let
    \[
    f (x, y) := 4 (1-x)(1-y) + \max \{ \lambda x^2, \mu y^2\}
    \]
    be a function in variables $x, y$.
    Then $f (x, y)$ has minimum value  $\lambda$ on the region $C:=\{(x,y)\in \mathbb{R}^2\ |\ 0\leq x,y\leq 1\}.$    
\end{lemma}
\begin{proof}
     Denote by $\theta = \sqrt{\lambda/\mu}$. Note that $2\sqrt{\frac{\lambda}{\mu}}=2\theta<\lambda+2\theta<2$. In particular, $\mu>\lambda$.
    Define the segments as follows: \begin{align*}
    L_1&:=\{(0,y)\in \mathbb{R}^2\ |\ 0\leq y\leq 1\},\\
    L_2&:=\{(x,1)\in \mathbb{R}^2\ |\ 0\leq x\leq 1\},\\
    L_3&:=\{(1,y)\in \mathbb{R}^2\ |\ \theta\leq y\leq 1\},\\
    L_4&:=\{(x,\theta x)\in \mathbb{R}^2\ |\ 0\leq x\leq 1\},\\
    L_5&:=\{(x,0)\in \mathbb{R}^2\ |\ 0\leq x\leq 1\},\\
    L_6&:=\{(1,y)\in \mathbb{R}^2\ |\ 0\leq y\leq \theta x\}.
    \end{align*}
    Let $R_1$ (resp. $R_2$) be the closed region enclosed by the edges $L_1, L_2, L_3$ and $L_4$ (resp. $L_4,L_5$, and $L_6$). It is easy to see 
    \[
    f(x,y)=\begin{dcases}
        4(1-x)(1-y)+\mu y^2 
         & \text{if $(x,y)\in R_1$}, \\
        4(1-x)(1-y)+\lambda x^2 & \text{if $(x,y)\in R_2$}.
    \end{dcases}
    \] Since $f(x,y)$ is a continuous function on the compact set $C=R_1\cup R_2$, there exists a point $Q=Q(x_1,y_1)\in C$ such that $f(Q)=\min_{P\in C} f(P)$ is the minimum value. If $Q$ is an interior point of $R_1$ (resp. $R_2$), then
    the partial differential $f_x=4(y-1)$ (resp. $f_y=4(x-1)$) vanishes at $Q(x_1,y_1)$ which is absurd. Thus, $Q\in L_i$ for some $i=1,2,3,4,5,6$. 

    In what follows, we shall show $f(Q)=\lambda$.
    Note that 
    \[
    f(x,y)=\begin{dcases}
         \mu>\lambda 
         & \text{if $(x,y)\in L_2$},\\
         \mu y^2\geq \lambda & \text{if $(x,y)\in L_3$}\\
         \lambda 
         & \text{if $(x,y)\in L_6$}.  
    \end{dcases}
    \]

    Suppose that $Q\in L_1$. As the restriction function $f|_{L_1}=f(0,y)$ has derivative $-4+2\mu y$ on $L_1$, by assumptions $\mu<2$, and $y_1\leq 1$, $Q$ must be one of the endpoints of $L_1$. So $f(Q)=f(0,1)=\mu>\lambda$ which is absurd.

    Suppose that $Q\in L_4$. As the function $f(x,y)|_{L_4}=f(x, \theta x)$ has derivative $2(\lambda+4\theta)x-4(1+\theta)$ on $L_4$, where $x\leq 1<\frac{2(1+\theta)}{\lambda + 4\theta}$ by the assumption,   $Q$ must be one of the endpoints of $L_4$. So $f(Q)=f(1,\theta)=\lambda$.
    
    Suppose that $Q\in L_5$. As the function $f(x,0)$ has derivative $-4+2\lambda x$ on $L_5$, by assumptions $\lambda<2$ and $x_1\leq 1$, $Q$ must be one of the endpoints of $L_5$. So $f(Q)=f(1,0)=\lambda$.
    The proof is complete. 
\end{proof}

\begin{lemma} \label{Lem:minfDE}
    Let $n \geq 4$ be an integer and let
    \[
    \begin{split}
        f_{D_n} (x, y) &:= 4 (1-x)(1-y) + \max \left\{\frac{1}{n-2} x^2, \frac{1}{n-3} y^2 \right\}, \\
        f_{E_n} (x, y) &:= 4 (1-x)(1-y) + \max \left\{\frac{9-n}{6(n-3)} x^2, \frac{9-n}{5(n-4)} y^2 \right\},
    \end{split}
    \]
    be functions in variables $x, y$.
    Then for $n \geq 4$ (resp.\ $n \in \{6, 7, 8\}$), we have 
    \[
    f_{D_n} (x, y) \geq \frac{1}{n-2} \quad \left(\text{resp.\ } f_{E_n} (x, y) \geq \frac{9-n}{6(n-3)} \right),
    \]
    for any $0 \leq x, y \leq 1$.
\end{lemma}

\begin{proof}
    We set $\lambda = 1/(n-2)$, $\mu = 1/(n-3)$ and $\theta = \sqrt{(n-3)/(n-2)}$.
    It is straightforward to check that the inequalities
    \begin{equation} \label{eq:funct}
        \lambda + 2 \theta < 2 \text{ and } \mu < 2
    \end{equation}
    for any $n \geq 4$.
    By Lemma~\ref{Lem:minfunc}, we have the result for $f_{D_n} (x, y)$.
    Similarly, by setting $\lambda = (9-n)/(6(n-3))$, $\mu =  (9-n)/(5(n-4))$ and $\theta =  \sqrt{5(n-4)/(6(n-3))}$, it is straightforward to check that the inequalities \eqref{eq:funct} hold for $n \in \{6, 7, 8\}$.
    Thus the assertion for $f_{E_n} (x, y)$ also follows from Lemma~\ref{Lem:minfunc}.
\end{proof}

%\textcolor{blue}{CHEN: It seems possible that one can give the similar inequalities for non-Gorenstein case by using the same method as in Theorems \ref{ineqDE} and \ref{3foldlocalineq}. However, it is harder to find the applications.}
\begin{theorem} \label{ineqDE}
    Let $P \in S$ be a germ of Du Val singularity of type $D_n, E_6, E_7$ or $E_8$, where $n \geq 4$.
    Let $\mathcal{L}$ be a mobile linear system on $S$ and let $\mu$ be a positive rational number such that the pair $(S, \frac{1}{\mu} \mathcal{L})$ is not log canonical.
    Then,
    \[
    \mathcal{L}^2 > \begin{dcases}
        \frac{1}{n-2} \mu^2, & \text{if $P \in S$ is of type $D_n$}, \\
        \frac{1}{6} \mu^2, & \text{if $P \in S$ is of type $E_6$}, \\
        \frac{1}{12} \mu^2, & \text{if $P \in S$ is of type $E_7$}, \\
        \frac{1}{30} \mu^2, & \text{if $P \in S$ is of type $E_8$}.
    \end{dcases}
    \]
\end{theorem}

%\begin{theorem} \label{ineqE}
    %Let $P \in S$ be a germ of Du Val singularity of type $E_n$, where $n = 6, 7, 8$.
    %Let $\mathcal{L}$ be a mobile linear system on $S$ and let $\mu$ be a positive rational number such that the pair $(S, \frac{1}{\mu} \mathcal{L})$ is not log canonical.
    %Then, \[ \mathcal{L}^2 \geq 
    %\begin{dcases}
     %   \frac{1}{6} \mu^2, & \text{if $n = 6$}, \\
     %   \frac{1}{12} \mu^2, & \text{if $n = 7$}, \\
     %   \frac{1}{30} \mu^2, & \text{if $n = 8$}.
    %\end{dcases}
    %\]
%\end{theorem}

%\noindent\textcolor{blue}{CHEN: Would it be better if we combine Theorems \ref{ineqD} and \ref{ineqE} into a Theorem as in Theorem \ref{3foldlocalineq}? }\textcolor{red}{OKADA:Yes! Please combine them.}

\begin{proof}[Proof of Theorem ~\ref{ineqDE}]
    Let $P \in S$ be a germ of Du Val singularity of type $D_n$ (resp.\ $E_n$), where $n \geq 4$ (resp.\ $n \in \{6, 7, 8\}$), and let $\pi \colon T \to S$ be the minimal resolution.
    There are $n$ prime exceptional divisors whose configuration is explained in their dual graphs:
        %%%%%%%%%%%%%%%%%%%%
    \iffalse
    \[
    \begin{split}
        D_n:& \quad \dynkin[radius=7mm, edgelength=0.7cm, labels={E_1,E_2,E_{n-1},E_{n-2},E_{n-1},E_n}, label directions={, , ,right, , }, scale=1.8]{D}{} \\
        E_n:& \quad \dynkin[radius=7mm, edgelength=0.7cm, make indefinite edge={5-6}, labels={E_1, E_n, E_2, E_3, E_4, E_{n-1}}, scale=1.8]{E}{6}
    \end{split}
    \]
    \fi
    %%%%%%%%%%%%%%%%%%%% 
    
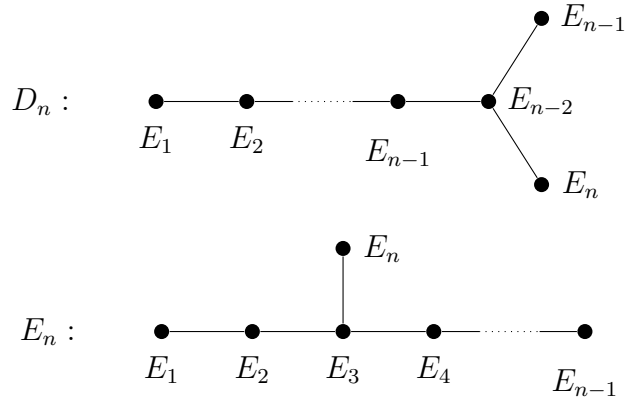
\begin{center}
% --- D_n Diagram ---
\begin{tikzpicture}[scale=1, every node/.style={circle, fill=black, inner sep=2pt}]
  \node[label=left:{$D_n : \quad$}, draw=none, fill=none] at (-0.5, 0) {};
  
  \node[label=below:{$E_1$}] (D1) at (0, 0) {};
  \node[label=below:{$E_2$}] (D2) at (1.2, 0) {};
  \node[label=below:{$E_{n-1}$}] (Dn1) at (3.2, 0) {};
  \node[label=right:{$E_{n-2}$}] (Dn2) at (4.4, 0) {};
  \node[label=right:{$E_{n-1}$}] (Dtop) at (5.1, 1.1) {};
  \node[label=right:{$E_n$}] (Dbot) at (5.1, -1.1) {};

  \draw (D1) -- (D2);
  \draw (D2) -- (1.8, 0);
  \draw[dotted] (1.8, 0) -- (2.6, 0);
  \draw (2.6, 0) -- (Dn1);
  \draw (Dn1) -- (Dn2);
  \draw (Dn2) -- (Dtop);
  \draw (Dn2) -- (Dbot);
\end{tikzpicture}

% --- E_n Diagram ---
\begin{tikzpicture}[scale=1, every node/.style={circle, fill=black, inner sep=2pt}]
  \node[label=left:{$E_n : \quad$}, draw=none, fill=none] at (-0.5, 0) {};

  \node[label=below:{$E_1$}] (E1) at (0, 0) {};
  \node[label=below:{$E_2$}] (E2) at (1.2, 0) {};
  \node[label=below:{$E_3$}] (E3) at (2.4, 0) {};
  \node[label=below:{$E_4$}] (E4) at (3.6, 0) {};
  \node[label=below:{$E_{n-1}$}] (En1) at (5.6, 0) {};
  \node[label=right:{$E_n$}] (Etop) at (2.4, 1.1) {};

  \draw (E1) -- (E2);
  \draw (E2) -- (E3);
  \draw (E3) -- (E4);
  \draw (E4) -- (4.2, 0);
  \draw[dotted] (4.2, 0) -- (5.0, 0);
  \draw (5.0, 0) -- (En1);
  \draw (E3) -- (Etop);
\end{tikzpicture}
\end{center}

    Note that $E_i^2 = -2$ for any $i$ and $E_i \cdot E_j = 1$ for any $i \neq j$ such that the vertices corresponding to $E_i$ and $E_j$ are connected by an edge.
    Let $\mathcal{L}$ be a mobile linear system on $S$ and write $\pi^*\mathcal{L} = \tilde{\mathcal{L}} + \sum a_i E_i$, where $\tilde{\mathcal{L}} := \pi_*^{-1} \mathcal{L}$ is the proper transform of $\mathcal{L}$.
    We set
    \[
    q (a_1, \dots, a_n) := 
    \begin{cases}
        q_{D_n} (a_1, \dots, a_n), & \text{if $P \in S$ is of type $D_n$}, \\
        q_{E_n} (a_1, \dots, a_n), & \text{if $P \in S$ is of type $E_n$},
    \end{cases}
    \]
    where $q_{D_n}$ and $q_{E_n}$ are the quadratic forms given in Lemmas~\ref{Lem:quadformD} and \ref{Lem:quadformE}, respectively.
    Then, we have
    \[
    \tilde{\mathcal{L}}^2 = \left(\pi^*\mathcal{L} - \sum_{i=1}^n a_i E_i\right)^2 
    = \mathcal{L}^2 - 2 q (a_1, \dots, a_n).
    \]
    We set
    \[
    r := 
    \begin{dcases}
        \frac{1}{n-2}, & \text{if $P \in S$ is of type $D_n$}, \\
        \frac{9-n}{6(n-3)}, & \text{if $P \in S$ is of type $E_n$}.
    \end{dcases}
    \]
    The aim is to show that $\mathcal{L}^2 > r \mu^2$.
    
    We have
    \[
    K_T + \frac{1}{\mu} \tilde{\mathcal{L}} + \sum_{i=1}^n \frac{a_i}{\mu} E_i = \pi^*\left(K_S + \frac{1}{\mu} \mathcal{L} \right).
    \]
    By the assumption that the pair $(S, \frac{1}{\mu} \mathcal{L})$ is not log canonical, one of the following takes place:
    \begin{enumerate}
        \item[(i)] $a_i/\mu > 1$ for some $i$; 
        \item[(ii)] The pair $(T, \frac{1}{\mu} \tilde{\mathcal{L}} + \sum \frac{a_i}{\mu} E_i)$ is not log canonical at a point $\tilde{P} \in E_i$, and $\tilde{P} \notin E_j$ for any $j \neq i$;
        \item[(iii)] The pair $(T, \frac{1}{\mu} \tilde{\mathcal{L}} + \sum \frac{a_i}{\mu} E_i)$ is not log canonical at a point $\tilde{P} \in E_i \cap E_j$ for some $i \neq j$.
    \end{enumerate}
    
    Suppose that we are in case (i).
    Then, by Lemmas~\ref{Lem:quadformD} and \ref{Lem:quadformE}, we have 
    \[
    \mathcal{L}^2 = \tilde{\mathcal{L}}^2 + 2 q (a_1, \dots, a_n) \geq \tilde{\mathcal{L}}^2 + r a_i^2 > r \mu^2
    \]
    since $\tilde{\mathcal{L}}^2 \geq 0$ and $a_i/\mu > 1$.
    
    Suppose that we are in Case (ii).
    We may assume that $a_i/\mu \leq 1$ for any $i$ because otherwise we are in Case (i).
    By the Corti's inequality \cite[Theorem~3.1]{Corti}, we have
    \[
    \tilde{\mathcal{L}}^2 > 4 \left(1 - \frac{a_i}{\mu}\right) \mu^2.
    \]
    We write $a_i = \alpha \mu$, where $0 \leq \alpha \leq 1$.
    Then,
    \[
    \begin{split}
        \mathcal{L}^2 &\geq \tilde{\mathcal{L}}^2 + r a_i^2 > 4 \left(1 - \frac{a_i}{\mu}\right) \mu^2 + r a_i^2 \\
        &=4(1-\alpha)\mu^2+r\alpha^2\mu^2=(\alpha-1)(r\alpha+r-4)\mu^2+r\mu^2 \\
        &\geq r \mu^2, 
    \end{split}
    \] since $\alpha\leq 1$ and $r\leq 2$.
    
    Suppose that we are in Case (iii), that is, the pair $(T, \frac{1}{\mu} \tilde{\mathcal{L}} + \sum \frac{a_i}{\mu} E_i)$ is not log canonical at $P \in E_i \cap E_j$.
    We may assume that $a_k/\mu \leq 1$ for any $k$.
    By the Corti's inequality \cite[Theorem~3.1]{Corti}, we have
    \begin{equation} \label{eq:pfsurf}
        \tilde{\mathcal{L}}^2 > 4\left(1-\frac{a_i}{\mu}\right)\left(1-\frac{a_j}{\mu}\right) \mu^2.
    \end{equation}
    We set $f (x, y) := f_{D_n} (x, y)$ (resp.\ $f_{E_n} (x, y)$) if $P \in S$ is of type $D_n$ (resp.\ $E_n$), where $f_{E_n} (x, y)$ and $f_{D_n} (x, y)$ are functions as in Lemma~\ref{Lem:minfDE}.
    By the inequality\eqref{eq:pfsurf} and by Lemmas~\ref{Lem:quadformD}, \ref{Lem:quadformE} and \ref{Lem:minfDE}, we have
    \[
    \begin{split}
        \mathcal{L}^2 &\geq \tilde{\mathcal{L}}^2 +
        \begin{dcases}
            \max \left\{ \frac{1}{n-2} a_i^2, \frac{1}{n-3} a_j^2 \right\}, & \text{if $P \in S$ if of type $D_n$}, \\
            \max \left\{ \frac{9-n}{6(n-3)} a_i^2,  \frac{9-n}{5(n-4)} a_j^2 \right\}, & \text{if $P \in S$ is of type $E_n$}.
        \end{dcases} \\
        &> f (a_i/\mu, a_j/\mu) \mu^2 \\
        &\geq r \mu^2.
    \end{split}
    \]
    This proves the assertions.
\end{proof}

\begin{remark}
    We can show that case (i) always takes place in the proof of Theorem~\ref{ineqDE}.
    The following argument is suggested by Ivan Cheltsov, which makes use of the fact that a Du Val singularity of type $D_n$ (for $n \ge 4$) and $E_n$ (for $n = 6, 7, 8$) is \textit{weakly-exceptional} (\cite[Example~4.7]{Pro00}), that is, it admits a unique \textit{plt blowup} (see for example \cite[Definitions~1.1 and 1.6]{Kud} for definitions of these notions).
    The argument is more or less the same as \cite[Lemma~2.5]{CK}. For readers convenience, we reproduce it as follows.

    Let $P \in S$, $\mathcal{L}$ and $\mu > 0$ be as in Theorem~\ref{ineqDE} and let $\pi \colon T \to S$ be the minimal resolution with exceptional divisors $E_1, \dots, E_n$ described as above.  
%    Let $m$ be a positive integer and let $\mathcal{L}^{(m)}$ be the movable linear system corresponding to $\operatorname{Sym}^m (H^0 (S, \mathcal{L}))$.
%    We have $(\mathcal{L}^{(m)})^2 = m^2 \mathcal{L}^2$ and, for any positive rational number $c$, the pair $(S, c \mathcal{L})$ is log canonical if and only if the pair $(S, \frac{c}{m} \mathcal{L}^{(m)})$ is log canonical.
%    By replacing $\mathcal{L}$ with $\mathcal{L}^{(m)}$ and $\mu$ with $m\mu$ for a sufficiently large $m > 0$, we may assume that $\mu > 1$.
    Let $c > 0$ be the log canonical threshold of the pair $(S, \mathcal{L})$, that is, the pair $(S, c \mathcal{L})$ is log canonical but not klt.
    Note that $c < 1/\mu$.
    Take a general member $D \in \mathcal{L}$ so that the pair $(S, c D)$ is log canonical but not klt.
    Let $\pi' \colon S' \to S$ be the partial resolution of singularities that contracts exactly one curve corresponding to the ``central" vertex of the Dynkin diagram of the minimal resolution $\pi$.
    Explicitly, $E'$ is the curve $E_{n-2}$ (resp.\ $E_3$) if $P \in S$ is of type $D_n$ (resp.\ of type $E_n$).
    Write 
    \[
    K_{S'} + c D' + c \gamma E' = {\pi'}^*(K_S + c D),
    \]
    where $\gamma \coloneq \ord_{E'} {\pi'}^*D$.
    We claim that $c \gamma = 1$.
    Assume not. 
    Then we have $c \gamma < 1$ and the pair $(S', c D' + E')$ is not log canonical since the pair $(S', c D' + c \gamma E')$ is log canonical but not klt.
    Hence the pair $(E', \Delta')$, where $\Delta' \coloneq cD'|_{E'} + \operatorname{Diff}_{E'} (0)$, is not log canonical.
    Moreover, we have $-(K_{E'} + \Delta') \sim_{\mathbb{Q}} - (1-c \gamma) E'|_{E'}$ is ample since $-E'$ is ample over $S$.
    By \cite[Theorem~2.1]{Kud}, this implies that the singularity $P \in S$ is not weakly-exceptional, which is a contradiction. 
    The claim is proved and we have $c \gamma = 1$.
    Let $E_i$ be the exceptional divisor corresponding to $E'$, that is, $i = n-2$ (resp.\ $i = 3$) if $P \in S$ is of type $D_n$ (resp.\ of type $E_n$).
    Then $a_i = \gamma$ and we have $a_i/\mu = \gamma/\mu = 1/(c \mu) > 1$.
    Thus, Case (i) always takes place. Note that this gives an alternate proof without using Corti's inequality \cite[Theorem~3.1]{Corti}, Lemmas \ref{Lem:minfunc} and \ref{Lem:minfDE}.
\end{remark}

\end{document}